\documentclass[12pt]{amsart}
\usepackage{amssymb}

\theoremstyle{plain}
\newtheorem{theorem}{Theorem}[section]
\newtheorem{corollary}[theorem]{Corollary}
\newtheorem{lemma}[theorem]{Lemma}
\newtheorem{proposition}[theorem]{Proposition}

\theoremstyle{definition}

\newcommand{\abs}[1]{\lvert#1\rvert}

\renewcommand{\le}{\leqslant}
\renewcommand{\ge}{\geqslant}
\renewcommand{\mid}{\::\:}

\newcommand\cstar{{\rm C}^*}                       
\newcommand\cstare{{\rm C}_{\rm e}^*}         

\def\rk{{\rm rank}\,}

\def\tr{\mathrm{tr}\,}

\def\bbC{\mathbb C}

\def\bbN{\mathbb N}

\def\bbR{\mathbb R}

\def\bbT{\mathbb T}

\def\cA{\mathcal A}
\def\cB{\mathcal B}

\def\cH{\mathcal H}

\def\cK{\mathcal K}
\def\cL{\mathcal L}

\def\cS{\mathcal S}
\def\cT{\mathcal T}

\DeclareMathOperator{\diag}{diag}

\DeclareMathOperator{\ran}{ran}
\DeclareMathOperator{\Res}{Res}

\begin{document}
\baselineskip 18pt

\title[Isometries of the Toeplitz Matrix Algebra]
{Isometries of the Toeplitz Matrix Algebra}

\author[D.~Farenick]{Douglas Farenick}
\address{Department of Mathematics and Statistics, University of Regina,
Regina, Saskatchewan S4S 0A2, Canada}
\email{douglas.farenick@uregina.ca}

\author[M.~Mastnak]{Mitja Mastnak}
\address{Department of Mathematics and Computing Science, Saint Mary's University,
Halifax, Nova Scotia B3H 3C3, Canada}
\email{mmastnak@cs.smu.ca}

\author[A.I.~Popov]{Alexey I. Popov}
\address{School of Mathematics and Statistics,
Newcastle University, Newcastle upon Tyne
NE1 7RU, United Kingdom}
\email{alexey.popov@ncl.ac.uk}

\thanks{This work supported in part by NSERC (Canada)}

\date{\today.}
 \begin{abstract}
 We study the structure of isometries defined on the algebra $\cA$ of upper-triangular Toeplitz matrices. Our first result is that a continuous multiplicative isometry $\cA\to M_n$ must be of the form either $A\mapsto UAU^*$ or $A\mapsto U\overline AU^*$, where $\overline A$ is the complex conjugation and $U$ is a unitary matrix. In our second result we use a range of ideas in operator theory and linear algebra to show that every linear isometry $\cA\to M_n(\bbC)$ is of the form $A\mapsto UAV$ where $U$ and $V$ are two unitary matrices. This implies, in particular, that every such an isometry is a complete isometry and that a unital linear isometry $\cA\to M_n(\bbC)$ is necessarily an algebra homomorphism.
 \end{abstract}

\maketitle

\section{Introduction}\label{introduction-section}

The $n\times n$ upper-triangular Toeplitz matrices over the field of complex numbers form a unital abelian subalgebra $\cA$
of the algebra $M_n(\bbC)$ of all $n\times n$ complex matrices. Our concern in this paper is with the structure of linear isometric maps
$\varphi:\cA\rightarrow M_n(\bbC)$, where the norm of a matrix $X\in M_n(\bbC)$ is the norm induced by considering $X$ as a linear operator on
the Hilbert space $\bbC^n$ with respect to the standard inner product. That there might be something of interest to deduce about such linear maps
is suggested by a result of Farenick, Gerasimova, and Shvai~\cite{FGS11} which arose from their
study of complete unitary-similarity invariants for certain complex matrices. Their result states that
if $\varrho:\cA\rightarrow M_n(\bbC)$ is a unital isometric homomorphism, then there is a unitary $U\in M_n(\bbC)$ 
such that $\varrho(X)=UXU^*$ for every $X\in\cA$.
In other words, every unital isometric homomorphism of the Toeplitz algebra $\cA$ back into $M_n(\bbC)$ extends
to an isometric automorphism of the algebra $M_n(\bbC)$. As a consequence of the results of the
present paper, this conclusion is also true for unital isometric maps that are merely linear. Hence, if a unital
linear map $\varphi:\cA\rightarrow M_n(\bbC)$ fails to be multiplicative, then the map cannot be an isometry. A similar conclusion is true for continuous maps that are multiplicative but not necessarily linear.

Every linear isometric map of an operator space into $M_n(\bbC)$ is completely bounded \cite[Proposition 8.11]{Paulsen-book},
but there are many examples of linear isometries that fail to be completely isometric---the transpose map
on $M_n(\mathbb C)$ being the most famous basic example. However, the restriction of the transpose map to $\cA$ is
completely isometric and it is a consequence of our work here
that every linear isometric map $\varphi:\cA\rightarrow M_n(\bbC)$
is completely isometric.  Thus, the results of this paper align with other results in which 
linear isometries of operator algebras are necessarily completely isometric (for example, the relevant results on
isometries of nest algebras and reflexive algebras in \cite{AS90, MT89, MT91, MQS94}).

There is a vast literature on the structure of maps defined on the algebra of complex $n\times n$ matrices that
preserve some properties of interest (such as the norm of a matrix, the spectrum, the rank, and so forth). A sample
list of papers devoted to ``preserver problems'' is \cite{BS00, CLS05, CL98, Kad51, LT92, LPS04, LSS03, M71, S81}.
Many such results depend on the use of matrix units or the abundance of rank-1 matrices in the full matrix algebra $M_n(\bbC)$.
Our contribution in this direction is rather novel in that we consider linear maps on a small subspace of matrices with limited structure
and which has just one (up to scalar multiple) rank-1 matrix and matrix unit.

Our main results in this paper are Theorem~\ref{multiplicative-structure} and Theorem~\ref{structure}. In the first result we show that every continuous multiplicative isometry $\cA\to M_n$ is of the form $A\mapsto UAU^*$ or of the form $A\mapsto U\overline{A}U^*$, where $U$ is a unitary matrix and $\overline{A}$ denotes the complex conjugation. In our second result we show that for every
linear isometry (not necessarily multiplicative) $\varphi:A\to M_n(\bbC)$ there exist two $n\times n$ unitary matrices $U$ and $V$ such that
$\varphi(A)=UAV$ for every $A\in\cA$. The proofs use a mix of algebra, matrix analysis, and operator theory.

Throughout the paper, we will use the symbol $S$ to denote the $n\times n$ nilpotent Jordan block of rank $n-1$:
\begin{equation}\label{eq:S}
S=\begin{bmatrix}
0 & 1 &        &        &   \\
  & 0 & 1      &        &   \\
  &   & \ddots & \ddots &   \\
  &   &        & 0      & 1 \\
  &   &        &        & 0
\end{bmatrix}
\end{equation}
(here, empty spaces mean zero entries). The Toeplitz matrix algebra $\cA$ consists of all matrices of
the form $f(S)$, where $f$ is an arbitrary complex polynomial. As $\cA$ contains the identity matrix $I$, the Toeplitz matrix
algebra is a unital operator algebra. The vector subspace $\cT=\cA+\cA^*$ of $M_n(\bbC)$
consists of all $n\times n$ Toeplitz matrices;
because $\cT$ contains the identity and is closed under the adjoint map $X\mapsto X^*$, the space $\cT$ is an
operator system \cite[Chapter 2]{Paulsen-book}. 

The norm $\|A\|$ of $A\in M_n(\bbC)$ is given by $\|A\|=\max\{\|Ax\|,:\,x\in\bbC^n,\;\|x\|=1\}$, where the norm of $x\in\bbC^n$
is the Hilbert space norm $\|x\|=\langle x,x\rangle^{1/2}$ and where $\langle \cdot,\cdot\rangle$ is the standard inner product on $\bbC^n$.
In contrast to the situation for Toeplitz operators acting on the Hardy space $H^2(\bbT)$, the exact determination of the
norm of a Toeplitz matrix is difficult, even in the case of $n=2$.

If $\cL\subset \cB(\cH)$ is a subspace, then a linear map $\varphi:\cL\rightarrow \cB(\cK)$ is said to be
completely contractive if the linear map $\varphi^{(k)}=\varphi\otimes{\rm id}_{M_k(\bbC)}:\cL\otimes M_k(\bbC)\rightarrow\cB(\cK)\otimes M_k(\bbC)$
is contractive for every $k\in\bbN$, and is completely isometric if every $\varphi^{(k)}$ is an isometry. (Here, $\cB(\cH)$ and $\cB(\cK)$
are the algebras of bounded linear operators acting on complex Hilbert spaces $\cH$ and $\cK$.) The map $\varphi$ is completely 
bounded if there is a $R>0$ such that $\|\varphi^{(k)}\|\leq R$ for all $k\in\mathbb N$. If $\cL$ contains the identity of $\cB(\cH)$, then
$\cL$ is called a unital operator space, and if a unital operator space $\cL$ is closed under the adjoint map, then $\cL$ is called an
operator system. Linear maps $\cL_1\rightarrow\cL_2$ of untial operator spaces that send the identity of $\cL_1$ to the
identity of $\cL_2$ are said to be unital.

Throughout this paper, $n$ shall remain fixed, $\cA$ shall always denote the unital, abelian subalgebra of $M_n(\bbC)$ consisting
of all upper-triangular Toeplitz matrices, and $\cT$ shall denote the operator subsystem of $M_n(\bbC)$ consisting of all Toeplitz matrices.


\section{Continuous Multiplicative Isometries}\label{multiplicative-section}

In this section we study isometric maps $\cA\to M_n(\bbC)$ which preserve the product of matrices but are not necessarily linear.

\begin{lemma}
Let $T\in M_n(\bbC)$ be a nilpotent matrix such that $||T||=||T^{n-1}||=1$.  Then $T$ is unitarily similar to $S$.
\end{lemma}
\begin{proof}
We can, up to a unitary similarity, assume that $T$ is strictly upper triangular.  Denote the super-diagonal entries of $T$ by $x_1, \ldots, x_{n-1}$.  Since $||T||=1$, we have that for each $i$, $|x_i|\le 1$.  Note that $T^{n-1}=x_1\ldots x_{n-1}S^{n-1}$ and conclude, using $||T^{n-1}||=1$, that $|x_1\ldots x_{n-1}|=1$.  Therefore $|x_1|=\ldots=|x_{n-1}|=1$.  Now a suitable diagonal unitary similarity (via $\operatorname{diag}\left(1, x_1, x_1x_2,\ldots, x_1x_2\ldots x_{n-1}\right)$) yields $T\sim S$.
\end{proof}

\begin{lemma} Let $\varphi\colon \cA \to M_n(\bbC)$ be a semigroup homomorphism such that $\varphi(S)=S$.  Then $\varphi(\cA)\subseteq \cA$.
\end{lemma}
\begin{proof} 
If $X\in\cA$, then $XS=SX$ implies that $\varphi(X)S=S\varphi(X)$ and so $\varphi(X)$ lies in the commutant of $S$, which is $\cA$.
\end{proof}

A norm preserving multiplicative map $\varphi\colon\cA\to M_n(\bbC)$ can in general be very pathological.  It does not even have to be homogeneous or skew-homogeneous: let $\cS=\{z\in\bbC:|z|=1\}$ denote the unit circle and let $\psi\colon \cS\to\cS$ be a group homomorphism (such maps can be, in case we do not demand continuity, very pathological).  Now define $\varphi_\psi\colon \cA\to\cA$ by $\varphi_\psi(0)=0$ and
$$\varphi_\psi\left(\sum_{i=r}^{n-1} a_i S^i\right) = \frac{|a_r|}{a_r}\psi\left(\frac{a_r}{|a_r|}\right)\sum_{i=r}^{n-1} a_i S^i,$$
where $a_r\not=0$.

Even if we assume that a multiplicative, norm preserving $\varphi\colon\cA\to\cA$ is $\mathbb{C}$-homogeneous we have the following non-continuous example: Fix $a\in \cS$ 
and define $\varphi$ by $\varphi(0)=0$ and
$$
\varphi\left(\sum_{i=r}^{n-1} \alpha_i S^i\right) = \sum_{i=r}^{n-1} a^{i-r}\alpha_i S^i
$$
where $\alpha_r\not=0$.

We now prove that if we additionally assume that $\varphi$ is continuous, then the number of choices become much smaller. We will prove that any continuous, norm preserving, multiplicative map $\varphi\colon\cA\to M_n(\bbC)$ is either a unitary similarity (i.e., $\varphi(T)=UTU^*$ for some fixed unitary $U$) or a complex-conjugate unitary similarity (i.e., $\varphi(T)=U\overline{T}U^*$ for some fixed unitary $U$).  

We will establish this claim by proving that any $\varphi\colon\cA\to \cA$ that is continuous, multiplicative, norm preserving, $\bbC$-homogeneous, and maps $S$ to $S$, must be the identity map.  The general claim then follows by invoking lemmas above together with the observation that a continuous, norm preserving map from $\cA$ to $\cA$ that maps $S$ to $S$ must either be homogeneous or skew-homogeneous.

\begin{lemma} Let $\varphi\colon\cA\to\cA$ be a continuous multiplicative map such that $\varphi(S)=S$ and for $\alpha\in\mathbb{C}$ we have $\varphi(\alpha I)=\alpha I$.  Then $\varphi$ is the identity map.
\end{lemma}
\begin{proof} We start by noting that the condition that  $\varphi(\alpha I)=\alpha I$ for $\alpha\in\mathbb{C}$ is equivalent to $\mathbb{C}$-homogeneity of $\varphi$.  For $i\in\mathbb{N}$ define
$$\cA_i=\mathrm{Span}\{S^j : i\le j\} = \left\{\sum_{j=i}^{n-1} \alpha_j S^j: \alpha_j\in\mathbb{C}\right\}.$$  
We point out that 
$$\cA_i = \{X\in\cA: S^{n-i}X=0\}.$$
If $A=\sum_{i=r}^{n-1} \alpha_i S^i$ with $\alpha_r\not=0$, and $\varphi(A)=B=\sum_{j=0}^{n-1} \beta_j S^j=\varphi(A)$, then we have $\beta_j=0$ for $j<r$ and $\beta_r=\alpha_r$.  This follows from:
\begin{eqnarray*}
0 &=& \varphi(0) = \varphi(S^{n-r}A) = S^{n-r} B = \sum_{j<r} \beta_j S^{n-r+j}, \\
\alpha_r S^{n-1} &=& \varphi(\alpha_r S^{n-1}) = \varphi(S^{n-1-r}A) = S^{n-r-1}B = \beta_r S^{n-1}.
\end{eqnarray*}
Now define $\psi\colon\cA_1 \to\cA$ by 
$$\varphi(I+N)=I+N+\psi(N)$$ 
for $N\in\cA_1$.  We will prove that $\psi=0$.  We first observe that for $i\in\mathbb{N}$ we have $\psi(\cA_i)\subseteq \cA_i$.  Indeed, if $N\in\cA_i$, then 
$S^{n-i}N=0$, thus 
\begin{eqnarray*}
S^{n-i} &=& \varphi(S^{n-i}) = \varphi(S^{n-i}(I+N)) = S^{n-i}(I+N+\psi(N)) = S^{n-i}+S^{n-i}\psi(N),
\end{eqnarray*}
and hence $S^{n-i}\psi(N)=0$.  Now let $r$ be the smallest positive integer such that $\psi(\cA_r)=0$ (we use the convention that $\cA_{n}=0$, so such an $r$ is well defined).  We will prove that $r=1$.  Suppose, toward contradiction, that $r>1$.  Let $T=\alpha S^{r-1}$ for some $\alpha\in\bbC$
and let $N\in\cA_r$.  Then we have the following identities:
\begin{eqnarray}
\psi(T+N)&=& (I+(I+T)^{-1}N)\psi(T) \label{eq-T+N}\\
\psi(2T)&=& 2(I+T)(I+(I+2T)^{-1}T^2)^{-1}\psi(T) \label{eq-2T}
\end{eqnarray}
The identity (\ref{eq-T+N}) is proven as follows.  Let $M=(I+T)^{-1}N$.  
Because  $N\in\cA_r$ and $(I+T)^{-1}=\sum_{k\geq 0}(-1)^k\alpha^kS^{k(r-1)}$, we see that $M\in\cA_r$
and, therefore, that $\psi(M)=0$.
Thus,
\begin{eqnarray*}
I+T+N+\psi(T+N)&=&\varphi(I+T+N) = \varphi((I+T)(I+M))\\
&=&\varphi(I+T)\varphi(I+M) = (I+T+\psi(T))(I+M)\\
&=&I+T+(I+T)M+(I+M)\psi(T)\\
&=& I+T+N+(I+(I+T)^{-1}N)\psi(T).
\end{eqnarray*}

Replacing $T$ by $2T$ and $N=T^2$ in identity (\ref{eq-T+N}) we then get
\begin{eqnarray*}
\psi(2T+T^2) = (I+(I+2T)^{-1}T^2)\psi(2T).
\end{eqnarray*}
Identity (\ref{eq-2T}) is proven by the following computation:
\begin{eqnarray*}
I+2T+T^2+(I+(I+2T)^{-1}T^2)\psi(2T) &=& I+2T+T^2+\psi(2T+T^2) \\
&=& \varphi(I+2T+T^2) = \varphi((I+T)^2) \\
&=& \varphi(I+T)^2 = (I+T+\psi(T))^2\\
&=& I+2T+T^2+2(I+T)\psi(T).
\end{eqnarray*}
We now use induction to prove that for every $m\in\mathbb{N}$ there is a polynomial $p_m$ such that
\begin{eqnarray}
\psi(2^m T) = 2^m(I+(2^m-1)T+p_m(T)T^2)\psi(T), \label{eq-2^mT}
\end{eqnarray}
The base case $m=1$ follows trivially from equation (\ref{eq-2T}) with 
\begin{eqnarray*}
p_1(T)&=& (I+T)(I+2T)^{-1} = (I+T)\sum_{i=0}^{n-1} (-2T)^i. 
\end{eqnarray*}
Now assume that the equation (\ref{eq-2^mT}) holds for some $m$.  For $m+1$ we then have
\begin{eqnarray*}
\psi(2^{m+1}T) &=& \psi(2(2^mT)) = 2(I+(2^mT)+p_1(2^m T)(2^m T)^2)\psi(2^m T) \\
&=& 2(I+2^m T + 2^{2m}p_1(2^m T)T^2)2^m(I+(2^m-1)T+p_m(T)T^2)\psi(T) \\
&=& 2^{m+1}(I+(2^m+2^m-1)T+(\mbox{other terms})T^2)\psi(T).
\end{eqnarray*}
This finishes the induction step.  

Using (\ref{eq-2^mT}) we 
now compute
\begin{eqnarray*}
\varphi\left(\frac{1}{2^m}+T\right) &=& \frac{1}{2^m}\varphi(I+2^m T) \\
&=& \frac{1}{2^m}\left(I+2^mT+\psi(2^mT)\right) \\
&=& \frac{1}{2^m} I + T + \psi(T)+ \left((2^m-1)I+p_m(T)T\right)T\psi(T).
\end{eqnarray*}
Take limit as $m\to\infty$ to get
\begin{eqnarray}
\varphi(T) = T+\psi(T)+\lim_{m\to\infty} ((2^m-1)I+p_m(T)T)T\psi(T).
\end{eqnarray}
Note that 
\begin{eqnarray*}
\lim_{m\to\infty} ((2^m-1)I+p_m(T)T)T\psi(T)
\end{eqnarray*}
can only exist if $T\psi(T)=0$: if $T\psi(T)=\sum_{j=i}^{n-1}\alpha_j S^j$ with $\alpha_i\not=0$, then $((2^m-1)I+p_m(T)T)T\psi(T)=\sum_{j=i}^{n-1} \beta_jS^j$ with $\beta_i=(2^m-1)\alpha_i\stackrel{m\to\infty}{\longrightarrow} \infty$.  So we have
$$
\varphi(T)=T+\psi(T),
$$
and therefore $\psi(T)=0$.  Now if $A\in\cA_{r-1}$, then we have $A=T+N$ for some $T\in\mathbb{C}S^{r-1}$ and $N\in\cA_r$ and by (\ref{eq-T+N}) we then have
\begin{eqnarray*}
\psi(A) = \psi(T+N) = (I+(I+T)^{-1}N)\psi(T) = 0.
\end{eqnarray*}
Hence $\psi(\cA_{r-1})=0$, contradicting minimality of $r$.

We have now established that $\psi(\cA_1)=0$.  Hence for every $N\in\cA_1$ we have
$$
\varphi(I+N) = I+N.
$$
Therefore for all $\alpha\not=0$:
$$
\varphi(\alpha I+N) = \varphi\left(\alpha\left(I+\frac{1}{\alpha} N\right)\right) =
\alpha\left(I+\frac{1}{\alpha}N\right) = \alpha I + N.
$$ 
Taking the limit as $\alpha\to 0$ we note that the above result is also valid for $\alpha=0$.  As every element of $\cA$ is of the form $\alpha I+N$ for some $\alpha\in\mathbb{C}$ and $N\in\cA_1$ we conclude that $\varphi$ is the identity map.
\end{proof}

\begin{lemma}Let $n>1$ and let $\varphi\colon\cA\to \cA$
be a continuous, multiplicative, norm preserving map, such that $\varphi(S)=S$.  Then $\varphi$ is either homogeneous or skew-homogeneous.
\end{lemma}
\begin{proof}
Let $\xi\colon\mathbb{C}\to\mathbb{C}$ denote the map given by $\varphi(\alpha S^{n-1})=\xi(\alpha) S^{n-1}$.  The map $\xi$ is well-defined (i.e., $\varphi(\alpha S^{n-1})$ is of required form) as $\varphi(S(\alpha S^{n-1}))=0=S\varphi(\alpha S^{n-1})$.  Due to norm preservation of $\varphi$ we have that $|\xi(\alpha)|=|\alpha|$.  Now let $\varphi(\alpha I) = \sum_{i=0}^{n-1} a_i S^{i}$.  Note that $a_0 S^{n-1} = S^{n-1}\sum_{i=0}^{n-1} a_i S^i=\varphi(S^{n-1}(\alpha I))=\xi(\alpha)S^{n-1}$.  So $a_0=\xi(\alpha)$.  Since $|a_0|=|\alpha|$ we conclude, by norm comparison, that $a_1=\ldots=a_{n-1}=0$.  Hence $\varphi(\alpha I)=\xi(\alpha)I$.

Let $\eta\colon \mathbb{C}\times\mathbb{C}\to\mathbb{C}$ denote the map given by $\varphi(\alpha I+\beta S^{n-1}) = \xi(\alpha)I+\eta(\alpha,\beta)S^{n-1}$; the fact that $\varphi(\alpha I+\beta S^{n-1})$ must be of this form is observed by considering 
$\xi(\alpha)S^i=\varphi(\alpha S^i)=\varphi(S^i(\alpha I+\beta S^{n-1}))=S^i\varphi(\alpha I+\beta S^{n-1}),$ for $i=1,\ldots ,n-1$.  Recall from \cite{FGS11} that
\begin{eqnarray*}
||\alpha I + \beta S^{n-1}||^2=\left|\left| \begin{pmatrix} \alpha & \beta \\ 0 & \alpha\end{pmatrix}\oplus \alpha I_{n-2}\right|\right|^2 =\left|\left| \begin{pmatrix} \alpha & \beta \\ 0 & \alpha\end{pmatrix}\right|\right|^2 = \frac{2|\alpha|^2+|\beta|^2+|\beta|\sqrt{4|\alpha|^2+|\beta|^2}}{2}.
\end{eqnarray*}  
Using $|\xi(\alpha)|=|\alpha|$ we conclude from the above that $|\eta(\alpha,\beta)|=|\beta|$.  The facts that for $\alpha,\alpha',\beta,\beta'\in\mathbb{C}$ we have
\begin{eqnarray*}
\varphi(\alpha'(\alpha I+\beta S^{n-1}))&=&\xi(\alpha')\varphi(\alpha I+\beta S^{n-1}),\\
\varphi(\beta S^{n-1})&=&\xi(\beta)S^{n-1},\mbox{ and}\\ 
\varphi((I+\beta S^{n-1})(I+\beta' S^{n-1}))&=& \varphi((I+\beta S^{n-1}))\varphi((I+\beta' S^{n-1}))
\end{eqnarray*} yield the identities
\begin{eqnarray}
\eta(\alpha,\alpha\beta) &=& \xi(\alpha)\eta(1,\beta), \\
\eta(0,\beta) &=& \xi(\beta),\\
\eta(1,\beta+\beta')&=&\eta(1,\beta)+\eta(1,\beta'),\label{eq8}
\end{eqnarray}
for all $\alpha,\alpha',\beta,\beta'\in\mathbb{C}$.
The equation (\ref{eq8}) tells us that 
$$\zeta:=\eta(1,-)\colon\mathbb{C}\to\mathbb{C}$$ is a continuous, additive, norm preserving map.  It is well known (and easy to deduce) that any such map is of the form $\zeta(z)=\lambda z$ or $\zeta(z)=\lambda \overline{z}$ for some fixed $\lambda\in\cS$.  We assume the former and will prove that this implies that $\varphi$ is homogeneous (a very similar consideration, which we leave to the reader, shows that the latter implies that $\varphi$ is skew homogeneous): Fix $\beta\not=0$ and compute
\begin{eqnarray*}
\xi(\beta)&=&\eta(0,\beta)=\lim_{\alpha\to 0} \eta(\alpha,\beta)\\ 
&=& \lim_{\alpha\to 0} \xi(\alpha)\eta(1,\beta/\alpha) = \lim_{\alpha\to 0} \frac{\xi(\alpha)}{\alpha} \lambda\beta \\
&=& \lambda\beta\lim_{\alpha\to 0} \frac{\xi(\alpha)}{\alpha}.
\end{eqnarray*}
As $\xi(1)=1$ we have that 
\begin{eqnarray*}
1= \lambda\lim_{\alpha\to 0} \frac{\xi(\alpha)}{\alpha},
\end{eqnarray*}
and therefore $\xi(\beta)=\beta$.
\end{proof}

Combining the results above leads directly to the following theorem, which is the main result of this section.

\begin{theorem}\label{multiplicative-structure}
Let $n>1$ and let
$$\varphi\colon\cA\to M_n(\mathbb{C})$$
be a continuous, multiplicative, norm preserving map.  Then there exists a unitary $U$ such that
either for all $A\in\cA$ we have
$$
\varphi(A)=UAU^*, \mbox{ or }
$$
for all $A\in\cA$ we have
$$
\varphi(A)=U\overline{A}U^{*}.
$$
\end{theorem}
\qed


\section{Singular Value Preservation}\label{singular-values-section}

In this section we start studying linear isometries on the algebra of Toeplitz matrices.


\begin{theorem}\label{singular-values-to-change}
If $\varphi:\cA\to M_n(\bbC)$ is a linear isometry, then $\varphi$ preserves singular values.
\end{theorem}
\begin{proof}
Our general approach will follow Morita's proof~\cite{Mor41} (see also~\cite[Theorem~10.2.2]{FM09}) of Schur's theorem~\cite{Sch25}.

For $A\in\cA$, consider the functions
$$
\phi(\lambda, A)=\det(\lambda I-A^*A)
$$
and
$$
\psi(\lambda,A)=\phi(\lambda,\varphi(A)).
$$
Each $A\in\cA$ can be written as $A=\sum_{k=1}^{n}x_kS^{k-1}+i\sum_{k=1}^{n}y_kS^{k-1}$ with $x_k,y_k\in\bbR$. Then both $\phi(\lambda,A)$ and $\psi(\lambda,A)$ can be thought of as polynomials (in $\lambda$) with coefficients in the ring $\bbC[x_1,\dots,x_n,y_1,\dots,y_n]$. Let us denote this ring by the shorter symbol $\bbC[x,y]$.

Notice that if $A\in\cA$ is fixed, then $\phi(\lambda,A)$ and $\psi(\lambda,A)$ considered as polynomials in $\bbC[\lambda]$ share a common root (the square of the norm of~$A$; this is because $\varphi$ is an isometry). We claim that, in fact, $\phi(\lambda,A)$ and $\psi(\lambda,A)$ considered as polynomials in $\bbC[x,y][\lambda]$, share a common factor of positive degree.

To that end, consider the \emph{resultant} $\Res(\phi,\psi)(\lambda)\in\bbC[x,y][\lambda]$ (see the books~\cite{Lang} or \cite{vdWae} for more information about the resultant. All we need to know is that if $f$ and $g$ are two polynomials with coefficients in an integral domain, then $\Res(f,g)$ is a new polynomial, constructed from $f$ and $g$ by a specific formula, with the property $\Res(f,g)=0$ if and only if $f$ and $g$ share a common factor of positive degree). If, again, $A\in\cA$ is fixed, then $\Res(\phi,\psi)$ becomes a polynomial in~$\bbC[\lambda]$. This polynomial must be the zero polynomial since, as mentioned above, $\phi$ and $\psi$ share a common root when $A\in\cA$ is fixed. Since this happens for all $A\in\cA$, it follows that $\Res(\phi,\psi)$ is, in fact, a zero polynomial when considered an element of $\bbC[x,y][\lambda]$.

We established that $\phi$ and $\psi$ have a common factor of positive degree. To finish the proof, it is enough to show that $\phi$ is, in fact, irreducible in $\bbC[x,y][\lambda]$ as that will mean that $\phi$ and $\psi$ coincide. 

Let us write 
\begin{equation}\label{eq:1}
\phi(\lambda)=f(\lambda)g(\lambda),
\end{equation}
where we treat $\phi(\lambda)$ as an element of $\bbC[x,y](\lambda)$, and assume that $f$ and $g$ both have positive degrees. For simplicity of notations, let us only consider matrices $A$ with imaginary part zero. If 
$$
A=\sum_{k=i}^nx_kS^{k-1},
$$
then
$$
A^*A=\begin{bmatrix}
x_1^2 & x_1x_2 & x_1x_3 & \dots & x_1x_n\\
x_1x_2 & x_1^2+x_2^2 & x_1x_2+x_2x_3 & \dots & x_1x_{n-1}+x_2x_n\\
x_1x_3 & x_1x_2+x_2x_3 & x_1^2+x_2^2+x_3^2 & \dots & x_1x_{n-2}+x_2x_{n-1}+x_3x_n\\
\vdots & \vdots & \vdots & \ddots & \vdots\\
x_1x_n & x_1x_{n-1}+x_2x_n & x_1x_{n-2}+x_2x_{n-1}+x_3x_n & \dots & x_1^2+x_2^2+\dots+x_n^2\\
\end{bmatrix}
$$
Write 
$$
f(\lambda)=p_0(x)+p_1(x)\lambda+\dots+p_{n-1}(x)\lambda^{n-1}+p_n(x)\lambda^n,
$$
$$
g(\lambda)=q_0(x)+q_1(x)\lambda+\dots+q_{n-1}(x)\lambda^{n-1}+q_n(x)\lambda^{n}
$$
and
$$
\phi(\lambda)=r_0(x)+r_1(x)\lambda+\dots+r_{n-1}(x)\lambda^{n-1}+r_n(x)\lambda^{n},
$$
where $r_n(x)=1$ (some of the coefficients in the above polynomials could, of course, be equal to zero). Recall that for each $k=0,\dots,n-1$, the coefficient $r_k(x)$ of $\phi(\lambda)=\det(\lambda I-A^*A)$ is equal to $(-1)^{n-k}$ times the sum of principal $(n-k)\times (n-k)$ minors of the matrix~$A^*A$. 

\emph{Claim 1}. For each $k\in\{0,1,\dots,n\}$, the polynomial $r_k(x)$ contains monomials $x_1^{2(n-k)},x_2^{2(n-k)},\dots,x_{k+1}^{2(n-k)}$ with non-zero coefficients. Conversely, if $r_k(x)$ has a monomial of the form $x_i^j$ with a non-zero coefficient, then $1\le i\le k+1$ and $j=2(n-k)$.

To see this, fix $1\le i_0\le n$ and let $x_i=0$ for $i\ne i_0$. This will make $A^*A$ into a diagonal matrix whose first $i_0-1$ diagonal entries are zero and the remaining diagonal entries are equal to~$x_{i_0}^2$. Then the principal minors are easy to calculate, and the Claim follows.

Now, write 
\begin{equation}\label{eq:2}
r_k(x)=p_0(x)q_k(x)+p_1(x)q_{k-1}(x)+\dots+p_{k-1}(x)q_1(x)+p_k(x)q_0(x).
\end{equation}
In particular, 
$$
r_0(x)=p_0(x)q_0(x)
$$
and
$$
r_1(x)=p_0(x)q_1(x)+p_1(x)q_0(x).
$$
From Claim~1, $r_1(x)$ contains $x_2^{2(n-1)}$ with a non-zero coefficient. Notice that $r_0(x)=(-1)^nx_1^{2n}$. It follows that either $p_0(x)$ or $q_0(x)$ must be a scalar. We may assume that $q_0(x)=1$. With this assumption, we have the following.

\emph{Claim~2}. For each $k\in\{0,1,\dots,n\}$, the polynomial $p_k(x)$ contains monomials $x_1^{2(n-k)},x_2^{2(n-k)},\dots,x_{k+1}^{2(n-k)}$ with non-zero coefficients. If $p_k(x)$ has a monomial of the form $x_i^j$ with a non-zero coefficient, then $1\le i\le k+1$ and $j\ge 2(n-k)$.

We prove Claim~2 by induction. The statement is true when $k=0$ since $p_0(x)=r_0(x)=(-1)^nx_1^{2n}$. Suppose it is valid for~$k$, we need to establish it for~$k+1$. Since $q_0(x)=1$, we get from~\eqref{eq:2} that
$$
p_k(x)=r_k(x)-p_0(x)q_k(x)-p_1(x)q_{k-1}(x)-\dots-p_{k-1}(x)q_1{x}.
$$
By Claim~1, $r_k(x)$ contains the monomials $x_1^{2(n-k)}, x_2^{2(n-k)},\dots,x_{k+1}^{2(n-k)}$ and does not contain any other monomials of the form $x_i^j$ with non-zero coefficients. By the induction hypothesis, if $x_i^j$ is a monomial in~$p_m(x)$ with a non-zero coefficient, where $m<k$, then $1\le i\le k$ and $j>2(n-k)$. In particular, $p_m(x)$ does not have a free coefficient. It follows that $x_1^{2(n-k)},x_2^{2(n-k)},\dots,x_{k+1}^{2(n-k)}$ will not be canceled, and these monomials are exactly the monomials of the form $x_i^j$ of the smallest degree in~$p_k(x)$. This proves Claim~2.

It follows from Claim~2 that $p_n(x)\ne 0$ and therefore $g(x)$ must have zero degree, a contradiction.
\end{proof}


The following is an immediate corollary of Theorem~\ref{singular-values-to-change}.
\begin{corollary}\label{local-decomposition}
If $\varphi:\cA\to M_n(\bbC)$ is a linear isometry, then for each $A\in\cA$ there exist two unitaries $U_A$ and $V_A$ such that $\varphi(A)=U_AAV_A$.
\end{corollary}


We remind the reader that the symbol $S$ was reserved to denote the nilpotent $n\times n$ Jordan block (see formula~\eqref{eq:S}).

\begin{corollary}\label{additional-observations}
If $\varphi:\cA\to M_n(\bbC)$ is a linear isometry,
then $\varphi(S^k)$ is a partial isometry for each~$k\ge 0$. By composing $\varphi$ with a multiplication by a unitary, we may arrange that $\varphi(I)=I$. 
\end{corollary}
\begin{proof}
This follows immediately from Corollary~\ref{local-decomposition}.
\end{proof}


\begin{corollary}\label{nilpotence-preservation}
If $\varphi:\cA\to M_n(\bbC)$ is a unital linear isometry, then $\varphi$ preserves the rank and the spectrum and maps nilpotent matrices to nilpotent matrices. 
\end{corollary}
\begin{proof}
It follows from Corollary~\ref{local-decomposition} that $\varphi$ preserves the rank of matrices in~$\cA$. Now, since $\varphi$ is unital and linear, we get for every $\lambda\in\bbC$ and $A\in\cA$:
$$
\lambda\in\sigma(A)\Longleftrightarrow\rk(\lambda I-A)<n\Longleftrightarrow\rk(\varphi(\lambda I-A))<n\Longleftrightarrow\lambda\in\sigma(A).
$$
Thus, $\varphi$ preserves the spectrum.
\end{proof}


To conclude this section, we record two auxiliary statements which will be heavily used throughout the rest of the paper.

\begin{proposition}\label{numerical-range}
Let $\cL$ be a space of $n\times n$ matrices containing the identity matrix and $\varphi:\cL\to M_n(\bbC)$ be a unital linear isometry. 
Then $\varphi$ preserves the numerical range of matrices.
\end{proposition}

\begin{proof} If $R\in\cL$, then the set $\{z\in\bbC\,:\,|\alpha z+\beta|\leq\|\alpha R+\beta I\|, \forall\,\alpha,\beta\in\bbC\}$ 
coincides the algebraic numerical range of $R$, which is the set of all complex numbers of the form 
$\psi(R)$, where $\psi:\cL\to\bbC$ is a linear functional such that $\|\psi\|=\psi(I)=1$ \cite[Chapter 1]{BD72}. 
Further, as each $R\in\cL$ is an operator
acting on a finite-dimensional Hilbert space (namely, $\bbC^n$), the algebraic numerical range coincides with the classical
numerical range of $R$, namely the set 
$W(R)=\{\langle R\xi,\xi\rangle\,:\,\xi\in\bbC^n,\;\|\xi\|=1\}$ \cite{BD72}. Thus,
\[
z\in W(R) \mbox{ if and only if }
|\alpha z+\beta|\leq\|\alpha R+\beta I\| \mbox{ for all }\alpha,\beta\in\bbC.
\]
Hence, $W(\varphi(R))=W(R)$ for every unital linear isometry $\varphi:\cL\to M_n(\bbC)$ and every $R\in\cL$.
\end{proof}


In the following statement, for a natural number $d$, we use the symbol $J_d$ to denote the $d\times d$ Jordan block with the zero diagonal. Notice that in this notation, $S=J_n$.

\begin{proposition}\label{fixed-point}
If $\varphi:\cA\to M_n(\bbC)$ is a untial linear isometry, $\varphi(S)=USU^*$ for some unitary matrix~$U$.
\end{proposition}
\begin{proof}
By Proposition~\ref{numerical-range}, the numerical range of $\varphi(A)$ coincides with the numerical range of $A$ for every $A\in\mathcal A$. 
In particular, the numerical ranges of $\varphi(S)$ and $S$ coincide. By Corollary~\ref{nilpotence-preservation}, $\varphi(S)$ is nilpotent and 
$\rk\big(\varphi(S)\big)=\rk(S)=n-1$.

We recall now a result of Haagerup and de la Harpe: if a Hilbert space operator $R$ is nilpotent of order $d$, then the numerical radius $w(R)\leq\|R\|\cos\frac{\pi}{d+1}$, and equality holds for a
contraction $R$ if and only if $J_d$ is a direct summand of $R$ \cite[Theorem 1(2)]{HdlH}. It follows that $\varphi(S)$ must contain, after a unitary similarity, a direct summand~$J_n$. Since the size of $J_n$ is~$n$, we get
$$
\varphi(S)=USU^*
$$
for some unitary matrix~$U$.
\end{proof}


The following proposition uses methods in operator algebras to describe the structure of the linear isometries from the algebra $\cA$ of upper-triangular Toeplitz matrices to $M_n(\bbC)$ under the additional assumption that the isometry is completely contractive. This proposition will be generalized in Section~\ref{the-structure-section} where the assumption that the map is completely contractive will be dropped. 

\begin{proposition} If $\varphi:\mathcal A\rightarrow M_n(\mathcal C)$ is a unital isometry such that $\varphi(S)$ is nilpotent
and $\varphi$ is completely contractive, then there exists a unitary $V$ such that $\varphi(A)=V^*AV$ for every $A\in\mathcal A$.
\end{proposition}

\begin{proof}
Consider the operator system $\mathcal S=\{A+B^*\,:\,A,B\in\mathcal A\}$ and the linear map $\tilde\varphi$ on $\mathcal S$
defined by 
\[
\tilde\varphi(A+B^*)=\varphi(A)+\varphi(B)^*.
\]
By \cite[Proposition 2.12, 3.5]{Paulsen-book}, $\tilde\varphi$ is well defined and is a completely positive 
and completely contractive linear extension of $\varphi$ to $\mathcal S$. By the Hahn-Banach theorem for completely
positive maps  \cite[Theorem 7.5]{Paulsen-book}, $\tilde\varphi$ admits a completely positive extension to $M_n(\mathbb C)$, which we denote by $\Phi$.
That is, $\Phi:M_n(\mathbb C)\rightarrow M_n(\mathbb C)$ is a unital completely positive linear map for which
$\Phi(A)=\varphi(A)$ for every $A\in \mathcal A$.

By Proposition~\ref{fixed-point}, $\Phi(S)=V^*SV$ for some unitary $V$, where $S$ is the $n\times n$ Jordan block with zero diagonal. Now consider the unital completely positive map $\Psi:M_n(\mathbb C)\rightarrow M_n(\mathbb C)$
defined by $\Psi(X)=V\Phi(X)V^*$. Thus, $S$ is an irreducible fixed point of $\Psi$. By Arveson's Boundary Theorem \cite[Theorem 2.1.1]{arveson1972}, \cite[Theorem 3.1]{farenick2011},
$\Psi(X)=X$ for all $X\in M_n(\mathbb C)$, which implies that $\varphi(A)=V^*AV$.
\end{proof}


\section{Spatial properties of the isometries}\label{spatial-properties-section}

The purpose of this section is to establish that if $\varphi:\cA\to M_n(\bbC)$ is a linear isometry, 
then the ranges and the kernels of $\varphi(S^k)$ are nested, where (as usual) $S$ is the $n\times n$ nilpotent Jordan block.

We start the section by recording the following useful lemma.


\begin{lemma}\label{nilpotence-examined}
Let $n\ge 2$, $A\in M_{n-1}(\bbC)$, $x,y\in\bbC^{n-1}$ and $\alpha\in\bbC$ are such that
$$
\rk\left(\begin{bmatrix}
x & (A-\lambda I_{n-1})\\
\alpha & y^T
\end{bmatrix}\right)=\rk(A-\lambda I_{n-1})=n-1
$$
for all $\lambda\ne 0$. Then $\alpha=0$, $A$ is nilpotent, and $y^TA^{k}x=0$ for all $k=0,1,\dots,n-1$.
\end{lemma}
\begin{proof}
First, observe that since $\rk(A-\lambda I_{n-1})=n-1$ for all $\lambda\ne 0$, the matrix $A$ must be nilpotent.

Fix a non-zero $\lambda\in\bbC$ and denote the matrix $A-\lambda I_{n-1}$ by~$B$. Consider the matrix
$$
T=\begin{bmatrix}
x & B\\
\alpha & y^T
\end{bmatrix}.
$$
Let $P$ be the cyclic permutation matrix and consider 
$$
R=T P=\begin{bmatrix}
B & x\\
y^T & \alpha 
\end{bmatrix}.
$$
Clearly $\rk R=\rk B=n-1$. Now, consider the product
$$
Q=\begin{bmatrix}
I_{n-1} & 0\\
-y^TB^{-1}\  & 1
\end{bmatrix}
\cdot
\begin{bmatrix}
B & x\\
y^T & \alpha 
\end{bmatrix}
=
\begin{bmatrix}
B & x\\
0^T & \alpha-y^TB^{-1}x
\end{bmatrix}.
$$
Since, obviously, $\rk Q=\rk R=n-1$ and $\rk B=n-1$, we must have $\alpha = y^TB^{-1}x$.

So, this shows that $\alpha = y^T(A-\lambda I_{n-1})^{-1}x$ for all $\lambda\ne 0$. Letting $\lambda\to\infty$, we get $\alpha = 0$. 

Writing $(A-\lambda I_{n-1})^{-1}$ as a Neumann series (see, e.g., \cite[Theorem~6.12]{AA02}) and using the fact that $A$ is nilpotent, we get, for all non-zero~$\lambda$,
$$
0=y^T(A-\lambda I_{n-1})^{-1}x=\frac{1}{\lambda}\sum_{k=0}^{n-1} \frac{1}{\lambda^k}y^TA^kx.
$$
Multiplying by~$\lambda$, we obtain
\begin{equation}\label{eq-1}
\sum_{k=0}^{n-1} \frac{1}{\lambda^k}y^TA^kx = 0;
\end{equation}
Letting $\lambda\to\infty$, we get $y^Tx=0$; multiplying~\eqref{eq-1} by~$\lambda$ and letting, again, $\lambda\to\infty$, we get $y^TAx=0$. Repeating $n-1$ times, we obtain $y^TA^kx=0$ for all $k=0,1,\dots,n-1$.
\end{proof}


The following theorem is the main statement of this section.

\begin{theorem}\label{nested}
If $\varphi:\cA\to M_n(\bbC)$ is a linear isometry, then
$$
\ker\big(\varphi(S)\big)\subseteq\ker\big(\varphi(S^k)\big)
$$
and
$$
\ran\big(\varphi(S)\big)\supseteq\ran\big(\varphi(S^k)\big)
$$
for all $k=1,2,\dots,n-1$.
\end{theorem}
\begin{proof}
By Corollary~\ref{additional-observations}, we may assume that $\varphi$ is unital. By Proposition~\ref{fixed-point}, there is a unitary $U_0$ such that $\varphi(S)=U_0SU_0^*$. Composing $\varphi$ with the unitary similarity by~$U_0^*$, we may assume without loss of generality that 
$$
\varphi(S)=S.
$$
For each $k=1,2,\dots,n-1$, let us use the notation
$$
T_k=\varphi(S^k).
$$
By Lemma~\ref{nilpotence-examined}, we have
$$
T_k=\begin{bmatrix}
x_k & A_k\\
0 & y_k^T
\end{bmatrix},
$$
where $A_k$ is a nilpotent $(n-1)\times(n-1)$-matrix and $y_k^Tx_k=0$.

Let $k$ be a fixed natural number between $2$ and~$n-1$. Since $\varphi(S+S^k) = S+T_k$, by Corollary~\ref{local-decomposition} there exist two $n\times n$ unitary matrices $U$ and $V$ such that 
\begin{equation}\label{eq:added}
S+T_k=U(S+S^k)V.
\end{equation}
Let $P$ denote the $n\times n$ cyclic permutation matrix,
$$
P=\begin{bmatrix}
0 &  & & 1\\
1	 & \ddots & & \\
  & \ddots & 0 & \\
 &  & 1 & 0
\end{bmatrix}
$$
(the empty spaces are filled with zeros). Consider the matrices $(S+S^k)P$ and $(S+T_k)P$. Let
$$
E=SP,\quad M=S^kP,\quad\mbox{and}\quad N=T_kP,
$$
so that $(S+S^k)P=E+M$ and $(S+T_k)P=E+N$. Observe that $E$ is the diagonal projection onto the span of the first $n-1$ basic vectors,
$$
E=\begin{bmatrix}
1 &        &   &   \\
  & \ddots &   &   \\
  &        & 1 &   \\
  &        &   & 0
\end{bmatrix},
$$
$M$ is a partial isometry of rank $n-k$ such that $M^*M\le E$ and $MM^*\le E$, and 
$$
N=\begin{bmatrix}
A_k & x_k\\
y_k^T & 0
\end{bmatrix}.
$$

In order to establish the proposition, we need to prove that $x_k=0$ and $y_k=0$. We will show that $y_k=0$; the statement about $x_k$ can be proven analogously.

So, let us assume that $y_k\ne 0$. Notice that by Corollary~\ref{additional-observations}, $N$ is a partial isometry. Let
$$
F=MM^*\quad\mbox{and}\quad G=NN^*,
$$
both $F$ and $M$ are orthogonal projections. Moreover, by Corollary~\ref{nilpotence-preservation}, $\rk(G)=\rk(N)=\rk(M)=\rk(F)$. Observe that the $(n,n)$-entry of $G$ is strictly positive, because $y_k\ne 0$.

Define now
\begin{align*}
A=(E+M)(E+M)^* & = E+ME+EM^*+MM^*\\%
& = E+M+M^*+F
\end{align*}
and
\begin{align*}
B=(E+N)(E+N)^* & = E+NE+EN^*+NN^*\\%
& = E+NE+EN^*+G.
\end{align*}
It is easy to see that, since $P$ is a unitary, $A=(S+S^k)(S+S^k)^*$ and $B=(S+T_k)(B+T_k)^*$. By the condition~\eqref{eq:added} this implies that the matrices $A$ and $B$ are unitarily similar. In particular,
$$
\tr(A^2)=\tr(B^2).
$$
We have:
\begin{align*}
A^2	& = (E+M+M^*+F)(E+M+M^*+F)= \\%
    & = E+EM+EM^*+EF   +   ME+M^2+MM^*+MF +\\
    & \quad\quad\  + M^*E+M^*M+M^{*2}+M^*F   +   FE+FM+FM^*+F,\quad\quad\quad\quad\quad\quad
\end{align*}
\begin{align*}
B^2	& = (E+NE+EN^*+G)(E+NE+EN^*+G)= \\%
    & = E+ENE+EN^*+EG   +   NE+NENE+NEN^*+NEG +\\
    & \quad\quad\  + EN^*E+EN^*NE+EN^*EN^*+EN^*G   +   GE+GNE+GEN^*+G.
\end{align*}
We claim that the trace of each summand in the expression for $A^2$ is always larger than or equal to the trace of the corresponding summand in the expression for~$B^2$. First, observe that the trace of every summand in the expression for $A^2$ is non-negative. Let us explore each summand. We have
\begin{equation}\label{eq-2}
  \tr(ENE)=\tr(NE)=\tr(A_k)=0,
\end{equation}
since the matrix $A_k$ is nilpotent. Similarly,
$$
\tr(EN^*)=0.
$$
Since $\rk(F)=\rk(G)$, $EF=F$, and the $(n,n)$-entry of $G$ is strictly positive, we have 
\begin{equation}\label{eq-2.1}
\tr(EG)<\tr(G)=\tr(F)=\tr(EF).
\end{equation}
Next, $\tr(NE)=0$, analogously to~\eqref{eq-2}. To find $\tr(NENE)$, observe that
$$
NE=\begin{bmatrix}
A_k & 0\\
y_k^T & 0
\end{bmatrix},
$$
so that
$$
NENE=\begin{bmatrix}
A_k & 0\\
y_k^T & 0
\end{bmatrix}
\cdot
\begin{bmatrix}
A_k & 0\\
y_k^T & 0
\end{bmatrix}
=
\begin{bmatrix}
A_k^2 & 0\\
* & 0
\end{bmatrix}.
$$
So, $\tr(NENE)=\tr(A_k^2)=0$, since $A_k$ is nilpotent. Next, for the summand $NEN^*$, observe that, as the diagonal entries of $N^*N$ are non-negative, we have
\[
\tr(NEN^*)=\tr(N^*NE)\le\tr(N^*N)=\rk(N^*N)=\rk(M^*M)=\tr(M^*M).
\]
For the remaining summands, using analogous considerations and the fact that $N$ is a partial isometry with $NN^*=G$, we obtain:\\
$\tr(NEG)=\tr(GNE)=\tr(NN^*NE)=\tr(NE)=0,\quad\mbox{by~\eqref{eq-2}};$\\
$\tr(EN^*E)=\tr(N^*E)=0,\quad\mbox{by~\eqref{eq-2}};$\\
$\tr(EN^*NE)=\tr(N^*NE)\le\tr(N^*N)=\rk(N^*N)=\rk(M^*M)=\tr(M^*M);$\\
$\tr(EN^*EN^*)=\overline{\tr(NENE)}=0$, by previous calculation;\\
$\tr(EN^*G)=\tr(EN^*NN^*)=\tr(EN^*)=\tr(N^*E)=0;$\\
$\tr(GE)=\tr(EG)<\tr(EF)=\tr(FE)$, by~\eqref{eq-2.1};\\
$\tr(GNE)=\tr(NEG)=0$, by previous calculation;\\
$\tr(GEN^*)=\overline{\tr(NEG)}=0$ because $E^*=E$ and $G^*=G$;\\
$\tr(G)=\rk(G)=\rk(F)=\tr(F)$.

These calculations establish our claim. Moreover, since some of the inequalities are strict, we derive that $\tr(B^2)<\tr(A^2)$, which is a contradiction. It follows that $y_k=0$ and hence $\ran\big(\varphi(S)\big)\supseteq\ran\big(\varphi(S^k)\big)$. 

The statement about the kernels is established analogously.
\end{proof}


\begin{corollary}\label{nested-explicit}
If $\varphi:\cA\to M_n(\bbC)$ is a linear isometry, then 
$$
\ker\big(\varphi(S)\big)\subsetneq\ker\big(\varphi(S^2)\big)\subsetneq\dots\subsetneq\ker\big(\varphi(S^{n-1})\big)
$$
and
$$
\ran\big(\varphi(S)\big)\supsetneq\ran\big(\varphi(S^2)\big)\supsetneq\dots\supsetneq\ran\big(\varphi(S^{n-1})\big).
$$
\end{corollary}
\begin{proof}
This follows from Theorem~\ref{nested} by using induction. The inclusions are strict because $\varphi$ preserves the rank of matrices.
\end{proof}


\section{The structure}\label{the-structure-section}

%
%
%
%

Finally, we arrive at the second main result of our paper.

\begin{theorem}\label{structure}
If $\varphi:\cA\to M_n(\bbC)$ is a linear isometry, then there exist two $n\times n$ unitaries $U$ and $V$ such that
\[
\varphi(A)=UAV
\]
for every $A\in\cA$.
\end{theorem}
\begin{proof}
As before, let us denote by $S$ the $n\times n$ Jordan block
\[
S=\begin{bmatrix}
0 & 1 & 0 & \dots & 0 \\
  & 0 & 1 &       & \\
\vdots  &   & \ddots &\ddots&    \\
  &&& 0&         1  \\
0  &&\dots&&         0  
\end{bmatrix}
\]
By Corollaries~\ref{additional-observations} and~\ref{nilpotence-preservation}, composing $\varphi$ with multiplications by appropriate unitary matrices on the left and on the right we may assume that $\varphi$ is unital and maps nilpotent matrices to nilpotent matrices. By Proposition~\ref{fixed-point}, further composing $\varphi$ with a unitary similarity  we may assume that $\varphi(S)=S$.

We will prove the statement of the theorem by induction on~$n$. If $n=1$, there is nothing to prove. Suppose that the statement has been proved for $n-1$, let us prove it for~$n$.

Let us denote the subalgebra in $\cA$ of all strictly upper triangular matrices by~$\cA_0$. Obviously, every matrix $A\in\cA$ can be written as $\alpha I+T$ for some $\alpha\in\bbC$ and $T\in\cA_0$. 

By Theorem~\ref{nested}, if $T\in\cA_0$, then the first column and the last row of $\varphi(T)$ are zero. Notice that a matrix $T\in\cA_0$ has the first column and the last row zero as well, and therefore it can be written as
\[
T=\begin{bmatrix}
0 & T_1\\
0 & 0^T
\end{bmatrix},
\]
where $T_1$ is an upper-triangular $(n-1)\times(n-1)$ Toeplitz matrix. We get:
\[
\varphi(T)=
\varphi(\begin{bmatrix}
0 & T_1\\
0 & 0^T
\end{bmatrix})=
\begin{bmatrix}
0 & *\\
0 & 0^T
\end{bmatrix},
\]
where $*$ stands for an unknown $(n-1)\times(n-1)$ matrix. This induces a linear isometry from the algebra of upper-triangular $(n-1)\times(n-1)$ matrices to $M_{n-1}(\bbC)$. Since $\varphi(S)=S$, this isometry is unital. By the induction hypothesis, there exists an $(n-1)\times(n-1)$ unitary $X$ such that
\[
\varphi(T)=
\varphi(\begin{bmatrix}
0 & T_1\\
0 & 0^T
\end{bmatrix})=
\begin{bmatrix}
0 & XT_1X^*\\
0 & 0^T
\end{bmatrix}.
\]
Define two $n\times n$ unitaries $U_0$ and $V_0$ by
\[
U_0=\begin{bmatrix}
X & 0\\
0^T & 1
\end{bmatrix}
\]
and
\[
V_0=\begin{bmatrix}
1 & 0^T\\
0 & X^*
\end{bmatrix}.
\]
So, if $T\in\cA_0$, then $\varphi(T)=U_0TV_0$, so that for an arbitrary $A\in\cA$ written as $A=\alpha I+T$, with $\alpha\in\bbC$ and $T\in\cA_0$, we have
\[
\varphi(\alpha I+T)=\alpha I+U_0TV_0.
\]
In what follows, we will show that $x=\gamma I_{n-1}$ for some scalar~$\gamma$. This will imply that $\varphi(T)=T$, finishing the proof.

To that end, consider the unital linear isometry $\psi:\cA\to M_n(\bbC)$ defined by
\[
\psi(A)=U_0^*\varphi(A)U_0.
\]
Denote the matrix $V_0U_0$ by~$W$. Then, if we write arbitrary $A\in\cA$ as $\alpha I+T$, with $\alpha\in\bbC$ and $T\in\cA_0$, we have
\[
\psi(\alpha I+T)=\alpha I+TW.
\]
Consider the space 
\[
\cL=\{TW\mid T\in\cA_0\}.
\]
Since $\psi$ is a unital isometry $\cA\to M_n(\bbC)$, it maps nilpotent matrices to nilpotent matrices by Corollary~\ref{nilpotence-preservation}. It follows that $\cL$ consists of nilpotent matrices. In particular,
\[
\tr((TW)^k)=0
\]
for every $T\in\cA_0$ and $k\ge 1$. Write $W=(w_{ij})$. Taking $T=S^{n-1}$ and $k=1$, we obtain
\[
w_{n1}=0.
\]
Considering $T=S^{n-2}$ with $k=1$, we get
\[
w_{n-1,1}+w_{n2}=0.
\]
Taking $k=2$ and using the fact that $w_{n1}=0$, we obtain
\[
w_{n-1,1}^2+w_{n2}^2=0.
\]
By \cite[Lemma 2.1.15(ii)]{RR00}, these two equalities imply
\[
w_{n-1,1}=w_{n2}=0.
\]
Similarly, taking $T=S^{n-3}$ with $k=1,2$, and $3$, we get
\[
w_{n-2,1}+w_{n-1,2}+w_{n3}=0,\quad w_{n-2,1}^2+w_{n-1,2}^2+w_{n3}^2=0\quad w_{n-2,1}^3+w_{n-1,2}^3+w_{n3}^3=0.
\]
Again, by \cite[Lemma 2.1.15(ii)]{RR00}, this means that 
\[
w_{n-2,1}=w_{n-1,2}=w_{n3}=0.
\]
Repeating the same procedure $n-1$ times shows that $W$ is, in fact, an upper-triangular matrix. Since $W$ is a unitary, its invariant subspaces are reducing, hence $W$ must be diagonal.

Write $W=\diag\{d_1,\dots,d_n\}$. Clearly $\abs{d_k}=1$ for all~$k$. Let us get back to the definition of $W=V_0U_0$. Writing $X=(x_{ij})_{i,j=1}^{n-1}$, we get:
\[
\begin{bmatrix}
1 & 0 & & 0 \\
0 & \bar x_{11} & \dots & \bar x_{n-1,1} \\
  & \vdots & & \vdots \\
0 & \bar x_{1,n-1} & \dots & \bar x_{n-1,n-1}
\end{bmatrix}
\cdot
\begin{bmatrix}
x_{11} & \dots & x_{1,n-1} & 0\\
\vdots & & \vdots & \\
x_{n-1,1} & \dots & x_{n-1,n-1} & 0\\
0 & & 0 & 1\\
\end{bmatrix}
=
\begin{bmatrix}
d_1 & & & \\
 & d_2 & & \\
 & & \ddots & \\
 & & & d_n
\end{bmatrix}.
\]
It is easy to see that this implies $X$ is diagonal, and $x_{11}=d_1$, $\bar x_{11}x_{22}=d_2$, $\bar x_{22}x_{33}=d_3$, $\dots$, $\bar x_{n-2,n-2}x_{n-1,n-1}=d_{n-1}$, $\bar x_{n-1,n-1}=d_n$. Denote the diagonal elements of $X$ by $u_1,\dots,u_{n-1}$. We obtain
\begin{align*}
u_{1}& =\, x_{11}=d_1, \\
u_{2}& =\, x_{22}=d_1d_2,\\
&\ \  \vdots\\
u_{n-1}& =\, x_{n-1,n-1}=d_1d_2\dots d_{n-1}=\bar d_n.
\end{align*}
Consider now $\tilde\varphi:\cA+\cA^*\to M_n(\bbC)$ defined by $\tilde\varphi(A+B^*)=\varphi(A)+\varphi(B)^*$. By \cite[Proposition~2.11]{Paulsen-book}, this is a well-defined positive map. Let us apply $\tilde\varphi$ to the matrix
\[
I+S+S^*+S^2+S^{*2}+\dots+S^{n-1}+S^{*n-1}=\begin{bmatrix}
1 & 1 & \dots & 1\\
1 & 1 & \dots & 1\\
\vdots & \vdots &  & \vdots\\
1 & 1 & \dots & 1\\
\end{bmatrix},
\]
which is positive. Since $\varphi(I)=I$, $\varphi(S)=S$, and $\varphi(\alpha I+T)=\alpha I+U_0TV_0$ (where $\alpha\in\bbC$ and $T\in\cA_0$), the image of this matrix under $\tilde\varphi$ is
\[
T_0:=\begin{bmatrix}
1 & 1 & u_1\bar u_2 & u_1\bar u_3 & \dots & u_1\bar u_{n-1}\\
1 & 1 & 1 & u_2\bar u_3 & \dots & u_2\bar u_{n-2}\\
\bar u_1u_2 & 1 & 1 & 1 & \dots & u_3\bar u_{n-3}\\
\bar u_1u_3 & \bar u_2u_3 & 1 & 1 & \dots & u_4\bar u_{n-4}\\
 & & \dots & & \\
\bar u_1u_{n-1} & \bar u_2u_{n-2} & \bar u_3u_{n-3} & \bar u_4u_{n-4} & \dots & 1
\end{bmatrix}
\]
Notice that $T_0$ must be positive. We will use now the following criterion for positivvity of operator matrices (see, e.g., \cite[Lemma~1.2]{Drit2004}): the operator matrix
\[
\begin{bmatrix}
RR^* & T^*\\
T & SS^*
\end{bmatrix}
\]
is positive if and only $T=SGR^*$ for some contraction~$G$. Let us apply this criterion to the upper-left $3\times 3$ corner of the matrix~$T_0$, which is
\[
\left[\begin{array}{cc|c}
1 & 1 & u_1\bar u_2\\
1 & 1 & 1 \\
\hline
\bar u_1u_2 & 1 & 1\\
\end{array}\right].
\]
We let $R=\begin{bmatrix}
1\\1
\end{bmatrix}$ and $S=\begin{bmatrix}
1
\end{bmatrix}$. The there must exists a $1\times 1$ contractive matrix $G=\begin{bmatrix}
g
\end{bmatrix}$ such that
\[
\begin{bmatrix}
\bar u_1u_2 & 1
\end{bmatrix}
=
\begin{bmatrix}
1
\end{bmatrix}
\cdot
\begin{bmatrix}
g
\end{bmatrix}
\cdot
\begin{bmatrix}
1 & 1
\end{bmatrix}
=
\begin{bmatrix}
g & g
\end{bmatrix}.
\]
It follows that $g=1$ and therefore $\bar u_1u_2=1$. This implies that 
\[
u_2=u_1.
\]
Repeating this procedure inductively for the upper-left $k\times k$ corner of~$T_0$ for $k=3,4,\dots,n$, we obtain $u_k=u_1$ for all $k=1,2,\dots,n-1$. It follows that $X=u_1I_{n-1}$, which proves the theorem.
\end{proof}


\begin{corollary}
Every linear isometry $\cA\to M_n(\bbC)$ is a complete isometry.
\end{corollary}


\begin{corollary}
If $\varphi:\cA\to M_n(\bbC)$ is a unital linear isometry,  
then $\varphi$ is an algebra homomorphism.
\end{corollary}


To conclude this section we consider now the 
operator system $\cT$ of all $n\times n$ Toeplitz matrices and determine the structure of all isometries $\varphi:\cT\rightarrow M_n(\bbC)$.

\begin{theorem} If $\cT$ is the subspace of $n\times n$ Toeplitz matrices, and if $\varphi:\cT\rightarrow M_n(\bbC)$ 
is a linear isometry, then there are unitaries $U, V\in M_n(\bbC)$ such that $\varphi(T)=UTV$ for every $T\in\cT$.
\end{theorem}

\begin{proof} As in the proof of Theorem~\ref{structure}, it is sufficient to prove that  
if $\varphi:\cT\rightarrow M_n(\bbC)$ 
is a unital linear isometry, then there is a unitary $U\in M_n(\bbC)$ such that $\varphi(T)=UTU^*$ for every $T\in\cT$.
Under this assumption, let
$\varphi_0:\cA\rightarrow M_n(\bbC)$ be the restriction of $\varphi$ to $\cA$. Hence, $\varphi_0$ is a unital linear isometry, and so
there is a unitary $U$ for which $\varphi(X)=UXU^*$ for every $X\in\cA$. Because $\cT=\cA+\cA^*$, the 
function $\Phi:\cT\rightarrow M_n(\bbC)$ given by $\Phi(X+Y^*)=\varphi(X)+\varphi(Y)^*$, for $X,Y\in\cA$, 
is a well-defined linear extension of $\varphi_0$ to
$\cT$ and satisfies $\Phi(T)=UTU^*$ for every $T\in\cT$. All that remains is to prove that $\varphi=\Phi$.

If $T\in\cT$ is hermitian, then its numerical range $W(T)$ is a subset of $\bbR$. Because $\varphi$ is a unital isometry, it preserves numerical
range and thus $W(\varphi(T))$ is also a subset of $\bbR$, which implies that $\varphi(T)$ is hermitian. Therefore, $\varphi$ is hermitian preserving,
which in turn implies that $\varphi(Y^*)=\varphi(Y)^*$ for every $Y\in\cT$. Hence, if $X,Y\in\cA$, then
$\Phi(X+Y^*)=\varphi(X)+\varphi(Y)^*=\varphi(X)+\varphi(Y^*)=\varphi(X+Y^*)$, which proves that $\varphi=\Phi$. 
\end{proof}

\section{Unital Linear Isometries into Arbitrary Matrix Algebras}

To this point we have considered only linear isometric maps of $\cA$ (or $\cT$) back into $M_n(\mathbb C)$. If
the codomain of $\varphi$ is a matrix algebra $M_m(\mathbb C)$ with $m\neq n$, then it is possible
to describe the structure of $\varphi$ in some special cases---for example, in cases where the unital linear map
$\varphi$ is known already to be completely isometric. Two key aspects of determining this structure are
the Arveson--Hamana theory of the C$^*$-envelope \cite[Chapter 15]{Paulsen-book} and Arveson's
description of operator systems of matrices and their C$^*$-envelopes \cite{arveson2010}. 

By Hamana's theorem \cite{hamana1979}, if $\cS\subseteq\cB(\cH)$ is a unital subspace, then there exists a pair $(\iota, \cstare(\cS))$ 
consisting of a unital C$^*$-algebra $\cstare(\cS)$ and a unital completely isometric linear
map $\iota:\cS\rightarrow\cstare(\cS)$ with the following properties:
\begin{enumerate}
\item $\iota(\cS)$ generates the C$^*$-algebra $\cstare(\cS)$;
\item for every unital completely isometric linear map $\kappa:\cS\rightarrow \cB(\cH_\kappa)$ there is a
surjective unital C$^*$-algebra homomorphism $\pi:\cstar\left(\kappa(\cS)\right)\rightarrow\cstare(\cS)$
such that the linear map $\pi\circ\kappa:\cS\rightarrow\cstare(\cS)$ is completely isometric.
\end{enumerate}
The C$^*$-algebra $\cstare(\cS)$ is called the C$^*$-envelope of $\cS$ and is unique up to isomorphism. It is useful
to note that $\cstare(\cS)=\cstare(\cS + \cS^*)$, and that $\cS+\cS^*$ is an operator system. 

Arveson's structure theory for matrix systems \cite{arveson2010}
states that if $\cS\subseteq M_m(\mathbb C)$ is a $d$-dimensional operator system of $m\times m$ matrices, then:
\begin{enumerate}
\item[{\rm (a)}] there exists $d$-dimensional involutive vector space $Z$
with distinguished unit $1\in Z$ and 
tuples $\Gamma=(\Gamma_1,\dots,\Gamma_p)$ and $\Omega=(\Omega_1,\dots,\Omega_q)$ 
of unital $*$-linear maps $\Gamma_k:Z\rightarrow M_{n_k}(\mathbb C)$ and 
$\Omega_j:Z\rightarrow M_{m_j}(\mathbb C)$
such that
\begin{enumerate}
\item[{\rm (i)}] the range of each $\Gamma_k$ and each $\Omega_j$ is an irreducible operator system,
\item[{\rm (ii)}] for every $r\in\{1,\dots,q\}$, $s\in\mathbb N$, and 
$[z_{ij}]_{i,j}\in M_s(Z)$,
\[
\left\|\left[\Omega_r(z_{ij})\right]_{i,j}\right\|
\,\leq\,
\max_{1\leq k\leq p} \left\|\left[\Gamma_k(z_{ij})\right]_{i,j}\right\|;
\]
\end{enumerate}
\item[{\rm (b)}] there is a unitary $W\in M_m(\mathbb C)$ such that $\cS=W(\cS_{\Gamma,\Omega})W^*$, where
\[
\cS_{\Gamma,\Omega}=\left\{ \left(\displaystyle\bigoplus_{k=1}^{p}\Gamma_k(z)\otimes I_{\ell_k}\right)
\,\bigoplus\,\left(\displaystyle\bigoplus_{j=1}^q \Omega_j(z)\otimes I_{i_j}\right)\,:\,z\in Z\right\},
\]
and where 
$X\otimes I_\ell$ denotes the direct sum of $\ell$ copies of a matrix $X$;
\item[{\rm (c)}] 
$\cstare(\cS)=\cstare(\cS_{\Gamma,\Omega})=\displaystyle\bigoplus_{k=1}^p M_{n_k}(\mathbb C)$.
\end{enumerate}

Note, in particular, that if $\cstare(\cS)$ is a simple C$^*$-algebra, then necessarily $p=1$ in the description above.

\begin{theorem}\label{representations-1} If $\varphi:\cA\rightarrow M_m(\mathbb C)$ is a unital linear isometry, then 
$m\geq n$. If, in addition, $\cstare\left(\varphi(\cA)\right)=M_n(\bbC)$, then there are 
a unitary $U\in M_m(\mathbb C)$, a positive integer $\ell$, and a unital
linear contraction $\psi:\cA\rightarrow M_{m-\ell n}(\mathbb C)$ such that
\[
\varphi(X) = U\left([X\otimes I_\ell]\oplus\psi(X)\right)U^*,
\]
for every $X\in\cA$.
\end{theorem}

\begin{proof} Because $\varphi$ is a unital isometry, $\varphi$ preserves numerical range. Thus, the numerical range
of $\varphi(S)$ is a circular disc about the origin of radius $\cos\frac{\pi}{n+1}$, which implies that $\varphi(S)$ is unitarily
equivalent in $M_m(\bbC)$ to a matrix of the form $S\oplus Y$ \cite[Theorem 1(2)]{HdlH}. Therefore, since $S\in M_n(\bbC)$ and $S\oplus Y\in M_m(\bbC)$, 
we must have $m\geq n$.

Suppose now that $\cstare\left(\varphi(\cA)\right)=M_n(\bbC)$.
Because $\varphi$ preserves numerical range, the map $\varphi$ is
hermitian preserving. Therefore, $\tilde\varphi(\cA+\cA^*)=\varphi(\cA)+\varphi(\cA)^*$, where $\tilde\varphi(X+Y^*)=\varphi(X)+\varphi(Y)^*$.
Noting that $\cT=\cA+\cA^*$, the unital linear map $\tilde\varphi:\cT\rightarrow M_m(\bbC)$ is a unital linear contraction
and $\tilde\varphi(\cT)=\varphi(\cA)+\varphi(\cA)^*$ is an operator system. By hypothesis, the C$^*$-envelope of 
$\tilde\varphi(\cT)=M_n(\bbC)$. Thus,
in Arveson's description of the operator system $\tilde\varphi(\cT)$ as $W(\cS_{\Gamma,\Omega})W^*$ above, 
it must be that $\Gamma$ is a $1$-tuple and that $\tilde\varphi(\cT)$ has the form
\[
\tilde\varphi(\cT)=\left\{ W\left( \left(\Gamma(z)\otimes I_\ell\right)\bigoplus \left(\bigoplus_{j=1}^q\Omega_j(z)\otimes I_{i_j}\right)\right)W^*\,:\,z\in Z\right\},
\]
for some involutive vector space $Z$ with distinguished unit $1\in Z$ and 
unital $*$-linear maps $\Gamma:Z\rightarrow M_{n_1}(\mathbb C)$ and 
$\Omega_j:Z\rightarrow M_{m_j}(\mathbb C)$ that have the properties [a(i)] and [a(ii)] indicated above.
In particular, because $M_n(\mathbb C)=\cstare(\varphi(\cA))=\cstare(\tilde\varphi(\cT))=\cstare(\cS_{\Gamma,\Omega})=M_{n_1}(\mathbb C)$,
we deduce that $n_1=n$. Thus, if $E$ is the block matrix 
\[
E=\left[ P\;\; 0_{1}\;\;\cdots\;\;0_{q}\right],
\]
where each $0_j$ denotes an $n\times (m_ji_j)$ matrix
of zeros and where $P$ is the $n\times (\ell n_1)$ matrix of the form $[I_n\;\; 0]$, 
then the linear map $\delta:\cA\rightarrow M_n(\mathbb C)$ given by
$\delta(X)=EW^*\varphi(X)WE^*$ is unital and contractive. If, given $X\in\cA$, 
$z$ is the unique element of $Z$ for which 
\[
\varphi(X)=W\left((\Gamma(z)\otimes I_\ell)\oplus(\Omega_1(z)\otimes I_{i_1})\oplus\cdots\oplus(\Omega_q(z)\otimes I_{i_q})\right)W^*,
\]
then $\delta(X)=\Gamma(z)$ and
\[
\|X\|=\|\varphi(X)\|=\max\left\{\|\Gamma(z)\|,\,\|\Omega_1(z)\|,\dots,\,\|\Omega_q(z)\|\right\}=\|\Gamma(z)\|=\|\delta(X)\|.
\]
(The second-to-last equality above is a result of 
property [a(ii)] in the case $s=1$.) Therefore, by Theorem~\ref{structure}, there is a unitary $V\in M_n(\mathbb C)$
such that $\delta(X)=VXV^*$ for every $X\in\cA$. Thus, 
$\varphi(X)=W\left( VXV^*\oplus\Omega_1(z)\oplus\cdots\oplus\Omega_q(z)\right)W^*$,
where $z$ is the unique element of $Z$ that yields $\varphi(X)$.
Because the map 
\[
X\mapsto (\Omega_1(z)\otimes I_{i_1})\oplus\cdots\oplus(\Omega_q(z)\otimes I_{i_q}),
\]
where $z$ is the unique element of $Z$ arising from $X\in\cA$,
defines a unital contractive linear map 
$\psi:\cA\rightarrow M_{m-n}(\mathbb C)$ by property [a(ii)] in the case $s=1$, in setting
\[
U=W\left((V\otimes I_\ell)\oplus (I_{m_1}\otimes I_{i_1})\oplus\dots\oplus (I_{m_q}\otimes I_{i_q})\right)
\]
we obtain a unitary $U$ for which $\varphi(X) = U\left([X\otimes I_\ell]\oplus\psi(X)\right)U^*$
for every $X\in \cA$.
\end{proof}

Theorem~\ref{representations-1} also admits the following formulation
in the category of unital operator spaces and unital completely contractive maps.

\begin{theorem}\label{representations-2} If $\varphi:\cA\rightarrow M_m(\mathbb C)$ is a unital completely 
isometric linear map, then $n\le m$ and there are 
a unitary $U\in M_m(\mathbb C)$, a positive integer $\ell$, and a unital
completely contractive map $\psi:\cA\rightarrow M_{m-\ell n}(\mathbb C)$ such that
\[
\varphi(X) = U\left([X\otimes I_\ell]\oplus\psi(X)\right)U^*,
\]
for every $X\in\cA$.
\end{theorem}

\begin{proof} As before, we note that $\varphi$ admits a unital completely contractive extension
$\tilde\varphi$ to the operator system $\cT=\cA+\cA^*$ via $\tilde\varphi(X+Y^*)=\varphi(X)+\varphi(Y)^*$,
for $X,Y\in\cA$.
Because $\varphi$ is unital and completely isometric, the operator spaces $\cA$
and $\varphi(\cA)$ have the same C$^*$-envelopes.  
However, because the C$^*$-algebra $\cstar(\cA)$ generated by $\cA$ is $M_n(\mathbb C)$, which is simple,
the C$^*$-envelope $\cstare(\cA)$ of $\cA$ necessarily coincides with $\cstar(\cA)$. 
Therefore, by Theorem~\ref{representations-1}, 
there are a unitary $U\in M_m(\mathbb C)$, a positive integer $\ell$, and a unital
linear contraction $\psi:\cA\rightarrow M_{m-\ell n}(\mathbb C)$ such that
\[
\varphi(X) =  U\left([X\oplus I_\ell] \oplus \psi(X)\right) U^*,
\]
for every $X\in\cA$. The proof of Theorem~\ref{representations-1} also shows that the map $ \psi$ is
constructed from the maps $\Omega_1,\dots,\Omega_q$, which implies that $\psi$ is completely contractive by 
property [a(ii)] and by the fact that $\varphi$ is completely isometric.
\end{proof}

Theorem~\ref{representations-2} above is predicted by a theorem of Blecher and 
Labuschagne \cite[Corollary 2.5(3)]{blecher--labuschagne2002}.
Although the proofs of Theorem~\ref{representations-2} and the Blecher--Labuschagne theorem 
are very different, both these results require, in one way or another, the Arveson--Hamana 
theory of the C$^*$-envelope.


%
%


\end{document}